\documentclass[J,11pt]{amsart}

\usepackage{amssymb, amsfonts, amsmath, amsthm, textcomp}

\theoremstyle{plain}
\newtheorem{thm}{Theorem}
\newtheorem{cor}{Corollary}

\newtheorem{proposition}{Proposition}

\let\lvert=|\let\rvert=|

\begin{document}

\title[Hyponormal Terraced Matrices with Transcendental Entries]{Fresh Examples of Hyponormal Terraced Matrices with Transcendental Entries}

\author[H. C. Rhaly Jr.]{H. C. Rhaly Jr.}
\address{
1081 Buckley Drive\\
Jackson, MS   39206\\ 
U. S. A.}
\email {rhaly@alumni.virginia.edu}

\keywords{Ces\`{a}ro operator, hyponormal operator, terraced matrix}

\subjclass[2010]{Primary 47B20}

\begin{abstract}  This note calls attention to an alternative version of the main result from [4], which can be used  together with Maclaurin series expansions and trigonometric identities to show that the terraced matrices generated by  the sequences $\{\ln (1+1/(n+1)) : n \geq 0\}$, $\{\tan (1/(n+2)) : n \geq 0\}$, and $\{\sinh (1/(n+2)) : n \geq 0\}$ are hyponormal operators on $\ell^2$.   Along the way it is also shown that the matrix generated by $\{\tan (1/(n+1)) : n \geq 0\}$ is {\it not} a hyponormal operator on $\ell^2$.
 \end{abstract}

\maketitle

\section{Introduction}

Assume that $a = \{a_n\}$ is a sequence of complex numbers such that the \textit{terraced matrix} $M= M (a)$, a lower triangular infinite matrix with constant row segments, acts through multiplication to give a bounded linear operator on  $\ell^2$, the Hilbert space of square-summable sequences.  \[M = M (a) = \left(\begin{array}{cccc} a_0 &0&0&\ldots\\a_1&a_1&0&\ldots\\a_2&a_2&a_2&\ldots\\\vdots&\vdots&\vdots&\ddots \end{array}\right).\]
(see [2, 3]). A \textit{hyponormal} operator $M$ satisfies   \[ \langle (M^*M - MM^*) f , f \rangle  \geq 0 \] for all $f$ in $\ell^2$.   Recall that the Ces\`{a}ro matrix $C$ is the terraced matrix that arises when $a_n = \frac{1}{n+1}$ for all $n \geq 0$.     In [1] it is shown that C is bounded, noncompact, and hyponormal; its norm is $|| C || = 2$, and its spectrum is \[\sigma(C) = \{\lambda :   \left|\lambda - 1\right| \leq 1 \}.\]    Since scalar multiples of hyponormal operators are hyponormal, $\pi C$ is an example of a hyponormal operator with transcendental entries.  However,  much more interesting examples are  possible, as we see in the next section.

\section{Main Results} 

Our primary tool is the main (and only) theorem in [4], where the proof can be found.

\begin{thm} Assume that $M :\equiv M (a) $ is a bounded linear operator on $\ell^2$ with $a_n> 0$ for all $n$, and $\{a_n\}$ is strictly decreasing to $0$.   If $0 < a_0 \leq 1$ and \[a_n(1-a_n) \leq a_{n+1} \leq \frac{a_n}{1+a_n}\] for each nonnegative integer $n$, then $M$ is hyponormal.
\end{thm}

The displayed inequality string from the theorem can be rearranged to obtain the following corollary, whose proof is obvious and will be omitted.  

\begin{cor} Assume that $M :\equiv M (a) $ is a bounded linear operator on $\ell^2$ with $a_n> 0$ for all $n$, and $\{a_n\}$ is strictly decreasing to $0$.   If $0 < a_0 \leq 1$ and \begin{equation} \frac{a_n-a_{n+1}}{a_n^2} \leq 1  \leq \frac{a_n-a_{n+1}}{a_na_{n+1}} \end{equation} for each nonnegative integer $n$, then $M$ is hyponormal.  \footnote{cf. [7, Theorem $6$].}
\end{cor}

\begin{proposition}  The terraced matrix generated by \[a_n = \ln (1 + \frac{1}{n+1})\] is a hyponormal operator on $\ell^2$. 

\end{proposition}

\begin{proof} We show that the inequality string (1) in Corollary $1$ is satisfied by looking at the upper inequality and the lower inequality separately and using the Maclaurin series expansion for $\ln (1+x)$.  Software such as [8] may be used to facilitate the computations.

(a) The upper inequality in (1) is satisfied when
\begin{equation} \frac{\ln\frac{(n+2)^2}{(n+1)(n+3)}}{(\ln \frac{n+2}{n+1})(\ln \frac{n+3}{n+2})} \geq 1.\end{equation} 
for all $n$.  For $N\geq 1$ take \[p_N(x) = \sum_{k=1}^N  (-1)^{k+1} \frac{x^k}{k} .\]  The inequality (2) above follows from the fact that \[\frac{p_4 (\frac{1}{(n+1)(n+3)})}{p_5 (\frac{1}{n+1}) \cdot p_5 (\frac{1}{n+2})} \geq 1 \] for all $n \geq 1$.   Inequality (2) can be verified separately for $n=0$.

(b)  The lower inequality in (1) is satisfied when \begin{equation} \frac{\ln\frac{(n+2)^2}{(n+1)(n+3)}}{(\ln \frac{n+2}{n+1})^2} \leq 1 \end{equation} for all $n$.  Inequality (3) follows from the fact that \[\frac{p_3 (\frac{1}{(n+1)(n+3)})}{(p_2 (\frac{1}{n+1}))^2} \leq 1 \] for all $n \geq 1$.  Inequality (3) can be verified separately for $n=0$. \end{proof}

\begin{cor}  The terraced matrix $M$ generated by \begin{equation*} a_n = \ln (1 + \frac{1}{n+k})  \end{equation*} is a  hyponormal operator on $\ell^2$ for all positive integers $k \geq 1$.
\end{cor}

\begin{proof}  This follows from [5].
\end{proof}

Does the approach of Proposition $1$ work when $a_n = \sin(\frac{1}{n+1})$ or $a_n= \tan^{-1} (\frac{1}{n+1})$?  It seems natural to ask this question since, in all three cases, $\{(n+1)a_n\}$ is strictly increasing to 1 -- and all three operators are {\it normaloid} (the norm equals the spectral radius; see [6]), as required for hyponormality.  However,  in these last two cases, the answer is -- No, since the upper inequality in Corollary $1$ is not satisfied, so the hyponormality question remains open for these two cases.

Note that $\{(n+1)a_n\}$ is strictly increasing in the three cases above. whereas $\{(n+1)a_n\}$ is strictly decreasing when $a_n = \tan(\frac{1}{n+1})$.

\begin{proposition}  The terraced matrix $M$ generated by \[a_n = \tan  (\frac{1}{n+1})\] is \textit{not} a hyponormal operator on $\ell^2$. 
\end{proposition} 

\begin{proof}   Since $lim_{n \to \infty} (n+1)a_n=1$ and $0 < a_0 < 2$, it follows from [$3$] that $\sigma(M) = \{\lambda :   \left|\lambda - 1\right| \leq 1 \}$.   If $\{e_n\}$ is the standard orthonormal basis for $\ell^2$, then  \[|| (M-I)e_0 ||^2  \geq  (\tan(1) -1)^2 + \sum_{n=1}^{\infty} (\frac{1}{n+1} + \frac{1}{3} (\frac{1}{n+1})^3)^2 \geq  \] \[  (\tan(1) - 1)^2 +\zeta(2)-1 + \frac{2}{3}( \zeta(4)-1) + \frac{1}{9}(\zeta(6)-1) > 1.01,\] so $|| M-I || > 1$, so $M-I$ is not normaloid and $M$ is not hyponormal.
\end{proof}

As we see below, the result is quite different for the terraced matrix generated by $a_n = \tan  (\frac{1}{n+2})$, and  it should be noted that  $\{(n+1)a_n\}$ is {\it not strictly decreasing} in this case.

\begin{proposition}  The terraced matrix $M$ generated by \[a_n = \tan  (\frac{1}{n+2})\] is a hyponormal operator on $\ell^2$.   
\end{proposition}

\begin{proof}
(a)  The upper inequality in Corollary $1$ is satisfied when \begin{equation} \frac{\sin \frac{1}{(n+2)(n+3)}}{\sin \frac{1}{n+2} \cdot \sin \frac{1}{n+3}} \geq 1.\end{equation} for all $n$. For $N\geq 0$ take \[r_N(x) = \sum_{k=0}^N  (-1)^{k} \frac{x^{2k+1}}{(2k+1)!} .\]  and \[s_N(x) = \sum_{k=0}^N  (-1)^{k} \frac{x^{2k}}{(2k)!} .\] Inequality ($4$) follows from the fact that \[\frac{r_1 (\frac{1}{(n+2)(n+3)})}{r_2 (\frac{1}{n+2}) \cdot r_2 (\frac{1}{n+3})} \geq 1 \] for all $n$.

(b)  The lower inequality in Corollary $1$ is satisfied when  \begin{equation} \frac{\sin \frac{1}{(n+2)(n+3)} \cdot \cos\frac{1}{n+2}}{\sin^2 \frac{1}{n+2} \cdot \cos \frac{1}{n+3}} \leq 1 \end{equation}  for all $n$.  Inequality ($5$)  follows from the fact that
\[\frac{r_0 (\frac{1}{(n+2)(n+3)}) \cdot s_0(\frac{1}{n+2})}{(r_1 (\frac{1}{n+2}))^2 \cdot s_1( \frac{1}{n+3})} \leq 1.\] for all $n$.   
\end{proof}

Next we consider $a_n = \sinh  (\frac{1}{n+1})$.  Since $a_0 > 1$, the hypothesis of Corollary $1$ is not satisfied, so hyponormality when $a_n = \sinh  (\frac{1}{n+1})$ remains an open question. But the issue for $a_n = \sinh  (\frac{1}{n+2})$ will be settled now.

\begin{proposition}  The terraced matrix $M$ generated by \[a_n = \sinh  (\frac{1}{n+2})\] is a hyponormal operator on $\ell^2$.   
\end{proposition}

\begin{proof}
(a)  The upper inequality in Corollary $1$ is satisfied when \begin{equation} \frac{2 \cosh (\frac{2n+5}{2(n+2)(n+3)}) \sinh {(\frac{1}{2(n+2)(n+3)})}}{\sinh (\frac{1}{n+2}) \sinh (\frac{1}{n+3})} \geq 1 \end{equation} for all $n$. For $N\geq 0$ take \[ r_N(x) = \sum_{k=0}^N \frac{x^{2k+1}}{(2k+1)!} .\]  and \[   s_N(x) = \sum_{k=0}^N   \frac{x^{2k}}{(2k)!} .\] Inequality ($6$) follows from the fact that \[  \frac{2 s_1 (\frac{2n+5}{2(n+2)(n+3)}) \cdot r_1 (\frac{1}{2(n+2)(n+3)} )}{(r_1 (\frac{1}{n+2})+2 \cdot \frac{1}{4!} (\frac{1}{n+2})^4) \cdot ( r_1 (\frac{1}{n+3})+2 \cdot \frac{1}{4!} (\frac{1}{n+3})^4) } \geq 1 \] for all $n$.

(b)  The lower inequality in Corollary $1$ is satisfied when  \begin{equation} \frac{2 \cosh (\frac{2n+5}{2(n+2)(n+3)}) \sinh {(\frac{1}{2(n+2)(n+3)})}}{( \sinh (\frac{1}{n+2}))^2 } \leq 1 \end{equation} for all $n$.  Inequality ($7$)  follows from the fact that
\[ \frac{2 (s_1 (\frac{2n+5}{2(n+2)(n+3)}) + \frac{2}{3!} (\frac{2n+5}{2(n+2)(n+3)})^3) (r_1 (\frac{1}{2(n+2)(n+3)}) + \frac{2}{4!} (\frac{1}{2(n+2)(n+3)})^4)}{(r_1 (\frac{1}{n+2}))^2} \leq 1\] for all $n$.   
\end{proof}

\begin{cor}  The terraced matrices $M$ generated by \begin{equation*} a_n = \tan  (\frac{1}{n+k})  \hspace{2mm} and   \hspace{2mm} a_n = \sinh  (\frac{1}{n+k})\end{equation*} are  hyponormal operators on $\ell^2$ for all positive integers $k \geq 2$.
\end{cor}

\section{Conclusion}  

In closing, we note the following questions, which have not been settled here. \vspace{1mm}

\noindent ($1$)  \hspace{1mm} Are the terraced matrices generated by the following sequences $\{a_n\}$ 
		
		\hspace{3mm}		  hyponormal operators on $\ell^2$?  

		 \hspace{6mm}		(a)   \hspace{1mm} $a_n = \sin(\frac{1}{n+1})$
		
		 \hspace{6mm}			(b)  \hspace{1mm}  $a_n= \tan^{-1} (\frac{1}{n+1})$  
		
		 \hspace{6mm}			(c) \hspace{1mm}  $a_n = \sinh(\frac{1}{n+1})$ 
		
		 \hspace{6mm}			(d)  \hspace{1mm}  $a_n= \sin^{-1} (\frac{1}{n+2})$     
		
 \vspace{1mm}
\noindent ($2$)   \hspace{1mm}   Does there exist a hyponormal terraced matrix generated by a 
		
				\hspace{3mm}	  positive sequence $\{a_n\}$ such that $\{(n+1)a_n\}$ is strictly decreasing?

\end{document}